\documentclass[preprint,12pt]{elsarticle}

\usepackage{amssymb,amsmath,amsthm}
\usepackage{amsfonts}
\usepackage{etoolbox}
\usepackage{multirow}
\usepackage{enumerate}
\usepackage{enumitem}
                                                                          
\makeatletter
\def\@author#1{\g@addto@macro\elsauthors{\normalsize%

    \def\baselinestretch{1}%
    \upshape\authorsep#1\unskip\textsuperscript{%
      \ifx\@fnmark\@empty\else\unskip\sep\@fnmark\let\sep=,\fi
      \ifx\@corref\@empty\else\unskip\sep\@corref\let\sep=,\fi
      }%
    \def\authorsep{\unskip,\space}%
    \global\let\@fnmark\@empty
    \global\let\@corref\@empty  
    \global\let\sep\@empty}%
    \@eadauthor={#1}
}
\patchcmd{\ps@pprintTitle}{\footnotesize\itshape
       Preprint submitted to \ifx\@journal\@empty Elsevier
       \else\@journal\fi\hfill\today}{\relax}{}{}
\makeatother

\theoremstyle{plain}

\newtheorem{theorem}{Theorem}

\newtheorem{lemma}[theorem]{Lemma}
\newtheorem{proposition}[theorem]{Proposition}
\newtheorem{conjecture}[theorem]{Conjecture}
\newtheorem{remark}[theorem]{Remark}
\newtheorem{question}[theorem]{Question}

\newtheorem{example}[theorem]{Example}

\newproof{pop1}{Proof of Proposition \ref{diamonds}}
\newproof{pop2}{Proof of Proposition \ref{lemseidel}}
\newproof{pot3}{Proof of Theorem \ref{main2}}

\newproof{pot4}{Proof of Proposition \ref{cmij}}

\begin{document}    
    \title{Matricial characterization of tournaments with maximum number of diamonds}
    
    \author[A]{Wiam Belkouche}
    \ead{belkouche.wiam@gmail.com}
    
    \author[A]{Abderrahim Boussa\"{\i}ri\corref{cor1}}  
    \cortext[cor1]{Corresponding author}
    \ead{aboussairi@hotmail.com}
    
    \author[A]{Soufiane Lakhlifi}
    \ead{s.lakhlifi1@gmail.com}
    
    \author[A]{Mohamed Zaidi}
    \ead{zaidi.fsac@gmail.com}

\address[A]{Facult\'e des Sciences A\"in Chock,  D\'epartement de Math\'ematiques et Informatique, Laboratoire de Topologie, Alg\`ebre, G\'eom\'etrie et Math\'ematiques Discr\`etes, Universit\'e Hassan II

Km 8 route d'El Jadida,
BP 5366 Maarif, Casablanca, Maroc}
        \begin{frontmatter}
        \begin{abstract}
             A \emph{diamond} is a $4$-tournament which consists of a vertex dominating or dominated by a $3$-cycle. 
						Assuming the existence of skew-conference matrices, we give a complete characterization of $n$-tournaments with the
						maximum number of diamonds when $n\equiv0\pmod{4}$ and $n\equiv3\pmod{4}$.
						For $n\equiv2\pmod{4}$, we obtain an upper bound on the number of diamonds
						in an $n$-tournament and we give a matricial characterization of
						tournaments achieving this bound.       
        \end{abstract}

     \end{frontmatter}
\section{Introduction}
One of the most important problems in Extremal Combinatorics is to determine
the largest or the smallest possible number of copies of a given object in a
finite combinatorial structure. We address
this problem in the case of tournaments. Throughout this paper, we mean by an 
\emph{$n$-tournament}, a tournament with $n$ vertices. It is easy to see that, up to
isomorphy, there are four distinct $4$-tournaments. The two that contain a
single $3$-cycle are called \emph{diamonds} \cite{bouchaala04, habib07, ille92}. They consist of a vertex
dominating or dominated by a $3$-cycle. The class of tournaments without
diamonds was characterized by Moon \cite{moon79}. These tournaments appear in
the literature under the names \emph{local orders} \cite{cameron81},\emph{
locally transitive tournaments} \cite{lachlan84} or \emph{vortex-free
tournaments} \cite{Knuth92}. For $n\geq 9$, Bouchaala \cite{bouchaala04}
proved that the number $\delta_{T}$ of diamonds  in an $n$-tournament is either
$0$, $n-3$, $2n-8$ or at least $2n-6$. 
In another side, motivated by geometric considerations, Leader and Tan \cite{leader10} proved 
that $\delta_T$ is at most $\frac{1}{4}\binom{n}{4}+O(n^3)$.
Moreover, by a probabilistic method, they showed
that there is an $n$-tournament with at least $\frac{1}{4}\binom{n}{4}$ diamonds.
To find the Tur\'{a}n density of a particular $4$-uniform hypergraphs,
Baber \citep{semeraro17} associates with each tournament $T=(V,A)$, the $4$-uniform hypergraph 
$\mathcal{H}_T$ on $V$ whose hyperedges correspond to subsets of $V$ which 
induce a diamond in $T$. In this hypergraph, every $5$-subset contains
either $0$ or $2$ hyperedges. Recently, using a combinatorial argument due to de Caen
\cite{decaen83}, Gunderson and Semeraro \cite{semeraro17} proved that an $r$-uniform
hypergraph in which every $(r+1)$-subset contains at most $2$ hyperedges has at most
$\frac{n}{r^2}\binom{n}{r-1}$ hyperedges, in particular,
an $n$-tournament contains at most $\frac{n}{16}\binom{n}{3}$ diamonds. Moreover, using Paley tournaments, they showed that this bound
is reached if $n=q+1$ for some prime power $q\equiv3\pmod{4}$. 

In this paper, we study the tournaments with the maximum number of diamonds. Our work is closely related to
the existence of D-optimal designs. More precisely, assuming the existence of skew-conference
matrices, we give a complete characterization of $n$-tournaments with the
maximum number of diamonds when $n\equiv0\pmod{4}$ and $n\equiv3\pmod{4}$.
For $n\equiv2\pmod{4}$, we obtain an upper bound on the number $\delta_T$
in an $n$-tournament. Moreover, we give a matricial characterization of
tournaments achieving this bound.  

\section{Number of diamonds and $3$-cycles in tournaments}
Throughout this paper, all matrices are from the set $\{0,\pm1\}$, unless otherwise noted. The identity matrix of order $n$ is denoted by $I_{n}$ and the all ones matrix is denoted by $J_n$. The absolute value $\left \vert M\right \vert $, of a matrix $M$, is obtained
from $M$ by replacing each entry of $M$ by its absolute value. Two matrices $A$ and $B$ are $\{\pm 1\}$-\emph{diagonally similar} if $B=DAD$ for some $\pm1$-diagonal matrix. 
 
Let $T$ be an $n$-tournament. With respect to a labelling, the \emph{adjacency matrix} of $T$
is the $n\times n$ matrix $A=(a_{ij})$ in which $a_{ij}$ is $1$ if $i$ dominates $j$ and $0$ otherwise. The
\emph{Seidel adjacency matrix} of $T$ is $S=A-A^{t}$ where
$A^{t}$ is the transpose of $A$. Remark that, with respect to different
labellings, the Seidel adjacency matrices of a tournament are permutationally similar.

\begin{remark}\label{detdiamond}
  The determinant of the Seidel adjacency matrix of a $4$-tournament is $9$ if it is a diamond and $1$ otherwise.
\end{remark}

The following lemma is crucial in our study.

\begin{lemma}
\label{diam principa minors}Let $T$ be an $n$-tournament and
 let $S$ be its Seidel adjacency matrix. Then the sum of all $4\times4$ principal
minors of $S$ is $8\cdot\delta_{T}+\binom{n}{4}$.
\end{lemma}

\begin{proof}
The number of $4\times4$ principal minors of $S$ is $\binom{n}{4}$. It follows
from Remark \ref{detdiamond} that the sum of all $4\times4$ principal minors of $S$ is 
$9\cdot\delta_{T}+(\binom{n}{4}-\delta_{T})=$ $8\cdot\delta_{T}+\binom{n}{4}$.
\end{proof}

Let $M$ be an $n\times n$ complex matrix and let $P_{M}(x):=\det(xI-M)=x^{n}%
+\sigma_{1}x^{n-1}+\cdots+\sigma_{n-1}x+\sigma_{n}$ be its
characteristic polynomial, then%

\begin{equation}
\sigma_{h}=(-1)^{h}\sum(\text{all }h\times h\text{ principal minors})\text{ }
\label{eq01}%
\end{equation}

When $M$ is a real skew-symmetric matrix, its nonzero eigenvalues are all
purely imaginary and come in complex conjugate pairs $\pm i\lambda_{1}%
,\ldots,\pm i\lambda_{k}$, where $\lambda_{1},\ldots,\lambda_{k}$ are real
positive numbers. Equivalently, the characteristic polynomial of $M$ has the form

\[
P_{M}(x)=x^{l}(x^{2}+\lambda_{1}^{2})(x^{2}+\lambda_{2}^{2})\cdots
(x^{2}+\lambda_{k}^{2})
\]

where $l+2k=n$.

Assume now that $M$ is skew-symmetric and all its off-diagonal entries are
from the set $\{-1,1\}$. Such matrix is sometimes known as a \emph{skew-symmetric Seidel matrix}.  By using \cite[Proposition~1]{mccarthy96},
$\det(S)=0$ if and only if $n$ is odd. Then, if $n$ is even, $l=0$ and
\[
P_{M}(x)=(x^{2}+\lambda_{1}^{2})(x^{2}+\lambda_{2}^{2})\cdots(x^{2}%
+\lambda_{n/2}^{2})
\]
If $n$ is odd, then by using \cite[Proposition~1]{mccarthy96} again, any
$(n-1)\times(n-1)$-principal minor is nonzero and thus, the multiplicity of
the eigenvalue $0$ is $1$. It follows that%

\[
P_{M}(x)=x(x^{2}+\lambda_{1}^{2})(x^{2}+\lambda_{2}^{2})\cdots(x^{2}%
+\lambda_{(n-1)/2}^{2})
\]

A useful formula of the number of diamonds is given in the following proposition.

\begin{proposition}\label{sumentries}
	Let $T$ be an $n$-tournament and let $S$ be its Seidel adjacency matrix. Then,
			\[
					\delta_{T}  = \frac{1}{96}n^2(n-1)(n-2) - \frac{1}{16}\sum_{i<j}{m_{ij}^2}
			\]
	where $m_{ij}$ are the entries of $S^2$.
\end{proposition}

\begin{proof}

Let $m$ be the integer part of $\frac{n}{2}$ and
$\pm i\lambda_{1},\ldots,\pm i\lambda_{m}$ the nonzero eigenvalues of $S$. As we have seen above
\begin{equation}
P_{S}(x)=%
\begin{cases}
(x^{2}+\lambda_{1}^{2})(x^{2}+\lambda_{2}^{2})\cdots(x^{2}+{\lambda_{m}^{2}%
}) & \text{ if }n\text{ is even }\\
x(x^{2}+\lambda_{1}^{2})(x^{2}+\lambda_{2}^{2})\cdots(x^{2}+{\lambda_{m}^{2}%
}) & \text{if }n\text{ is odd}%
\end{cases}\label{eq1}
\end{equation}

The nonzero eigenvalues of $S^2$ are $-\lambda_{1}^2,\ldots,-\lambda_{m}^2$,
each of them appears two times. Hence, we can write $P_{S^2}(x)$ in the following form.

\begin{equation}
	P_{S^2}(x)=%
	\begin{cases}
		(x^{2}+2\lambda_{1}^{2}x+\lambda_{1}^{4})\cdots(x^{2}+2\lambda_{m}^{2}x+\lambda_{m}^{4})%
		& \text{ if }n\text{ is even }\\
		x(x^{2}+2\lambda_{1}^{2}x+\lambda_{1}^{4})\cdots(x^{2}+2\lambda_{m}^{2}x+\lambda_{m}^{4})
		& \text{if }n\text{ is odd}%
	\end{cases}\label{eq2}
\end{equation}

Let $P_{S}(x):=x^n+\alpha_1 x^{n-1}+\ldots +\alpha_n$ and let $P_{S^2}(x):=x^n+\beta_1x^{n-1}+\ldots +\beta_n$.
By expanding expressions (\ref{eq1}) and (\ref{eq2}), we get
	\[
		\alpha_2 = \sum_i\lambda_{i}^{2}
	\]
	
	\[ 
		\alpha_4 = \sum_{i<j}{\lambda_i^2 \lambda_j^2} = \frac{1}{2}((\sum_i\lambda_{i}^{2})^{2}-\sum_i\lambda_{i}^{4})
 	\]
 	and
 	\[ 
		\beta_2 = 4\sum_{i<j}{\lambda_i^2 \lambda_j^2} + \sum_{i}{\lambda_i^4}
 	\]

	It follows that $\beta_{2}=2\alpha_{4}+\alpha_2^{2}$. By Equality (\ref{eq01}), we have $\alpha_2=\frac{n(n-1)}{2}$, and hence $\alpha_4 = \frac{1}{2}\beta_{2}-\frac{1}{2}\left(  \frac{n(n-1)}{2}\right)^{2}$.
	
	Since $S^2$ is symmetric and all its diagonal entries are $1-n$, by Equality (\ref{eq01}), we have
	
	\[ \beta_2 = \sum_{i<j}((n-1)^2 - m_{ij}^2) = \frac{n(n-1)^3}{2}-\sum_{i<j}m_{ij}^{2} \]
	
			Applying Lemma (\ref{diam principa minors}) and Equality (\ref{eq01}) again, we get		
	
								\[ \alpha_{4}= 8\delta_T+\binom{n}{4} \]
			It follows that
	\begin{eqnarray*}
		\delta_{T} & = & \frac{1}{8}\left(\alpha_{4}-\binom{n}{4}\right) \\
		           & = & \frac{1}{8}\left(  \frac{1}{2}\beta_{2}-\frac{1}{2}\left(  \frac{n(n-1)}{2}\right)^{2}-\binom{n}{4}\right) \\
		           & = & \frac{1}{16}\beta_{2}-\frac{1}{16}\left(  \frac{n(n-1)}{2}\right)^{2}-\frac{1}{8}\binom{n}{4} \\
		           & = & \frac{1}{96}n^2(n-1)(n-2) - \frac{1}{16}\sum_{i<j}{m_{ij}^2} 
	\end{eqnarray*}

	\end{proof}
	
	Let $T=(V, A)$ be a tournament, the \emph{switching} of $T$, according to a subset $X$ of $V$, consists of reversing all the arcs between $X$ and $V\backslash X$, we denote the resulting tournament by $T_{X}$. We say that two tournaments $T$ and $T^{'}$ on a vertex set $V$ are \emph{switching equivalent}, if there exists $X\subset V$ such that $T_{X} = T^{'}$. 
	\begin{remark} \label{remsw}
	It is well-known that two tournaments are switching equivalent iff their Seidel adjacency matrices are $\{\pm 1\}$-diagonally similar \cite{moorhouse95}. Since similarity by a $\pm1$ diagonal matrix preserves the principal minors, by Remark \ref{detdiamond}, switching equivalent tournaments have the same diamonds.
	\end{remark}
	Let $v$ be a vertex of a tournament $T=(V,A)$, \emph{the out-neighbourhood} $ N^{+}_T(v)$ of $v$ is the set of all vertices
of $T$ dominated by $v$. \emph{The in-neighbourhood} $ N^{-}_T(v)$ of $v$ is the set of all
vertices of $T$ which dominate $v$. In the switching $T_{N^{-}(v)}$ of $T$ according to $ N^{-}_T(v)$, the vertex $v$ dominates $V\setminus \{v\}$. Hence, by Remark \ref{remsw}, to study the number of diamonds, we can assume that there is a vertex $v$ dominating $V\setminus \{v\}$.
We obtain then the following lemma connecting the number of diamonds and the number of $3$-cycles.
\begin{lemma}\label{deltac3}
Let $T=(V, A)$ be a tournament containing a vertex $v$ that dominates $V\setminus \{v\}$. Then,
\[ \delta_T = \delta_{T - v} + c_3(T- v)\]
where $c_3(T-v)$ is the number of $3$-cycles in $T-v$.
\end{lemma}

Similarly to $\delta_T$, the number $c_3(T)$ of $3$-cycles in $T$ can also be expressed in terms of the entries of $S^2$.

\begin{proposition}\label{cmij}
		Let $T$ be an $n$-tournament and let $S$ be its Seidel adjacency matrix. Then,
			\[
					c_3(T)  = \frac{1}{24}n(n-1)(n-2) +\frac{1}{4}\sum_{i<j}{m_{ij}}
			\]
		where $m_{ij}$ are the entries of $S^2$.
\end{proposition}

Before proving this proposition, we need some notions and basic results about tournaments. 
For more details, the reader is referred to \cite{moon79}.

Let $T$ be an $n$-tournament. Without loss of generality, we can assume that the vertex set of $T$ is $\{1,\ldots,n\}$.
\emph{The out-degree} $d^{+}_T(i)$ (resp. \emph{in-degree} $d^{-}_T(i)$) of a vertex $i$ is $|N^{+}_T(i)|$ (resp. $|N^{-}_T(i)|$).
The out-degree $d^{+}_T(i,j)$ of $(i,j)$ (resp. in-degree $d^{-}_T(i,j)$ of $(i,j)$) is $|N^{+}_T(i)\cap N^{+}_T(j)|$ (resp. $|N^{-}_T(i)\cap N^{-}_T(j)|$).

The tournament $T$ is \emph{regular}, if there is a constant $k$ such that $d^{+}_T(i)=k $ for every $i\in \{1,\ldots,n\}$; it is \emph{doubly regular} if there is a constant $h$ such that  $d^{+}_T(i,j)=h$ for every $i\neq j\in \{1,\ldots,n\}$. A doubly regular tournament is also regular.
If $T$ is regular, then $n$ is odd and $d^{+}_T(i)=\frac{n-1}{2}$ for $i\in \{1,\ldots,n\}$. If $T$ is doubly regular, then $n\equiv 3 \pmod{4} $ and $d^{+}_T(i,j)=d^{-}_T(i,j)=\frac{n-3}{4}$ for $i\neq j\in \{1,\ldots,n\}$. 
	
Recall the well-known equalities
	
	\begin{equation}	
			d^{+}_T(i) + d^{-}_T(i) = n - 1
			\label{sumdeg}
	\end{equation}

  \begin{equation}	
			\sum_{i}{d^{+}_T(i)}=\sum_{i}{d^{-}_T(i)}
														 =\frac{n(n-1)}{2}
	\label{dx}
	\end{equation}
		\begin{equation}	
			c_3(T)= \binom{n}{3} - \sum_{i}\binom{d^{+}_T(i)}{2}
			\label{c3T}
	\end{equation}
	
	\begin{remark} \label{c3reg}
			It follows from Equality (\ref{c3T}) that $c_3(T)\leq \frac{1}{24}(n^3-n)$. Moreover equality holds iff $n$ is odd and $T$ is regular.
	\end{remark}

For $i\neq j \in \{1,\ldots,n\}$, let $\gamma_{ij}:=|N^{+}_T(i)\cap N^{-}_T(j)|+|N^{-}_T(i)\cap N^{+}_T(j)|$.
Then, we have

	\begin{equation}
		m_{ij} =2\gamma_{ij} - n + 2  
		\label{mij}
	\end{equation}	
	
		\begin{equation}
			d^{+}_T(i, j) - d^{-}_T(i, j) = d^{+}_T(i) - d^{-}_T(j) 
			\label{dij}
		\end{equation}
	
	\begin{equation}
	\gamma_{ij} = n - 2 - (d^{+}_T(i, j) + d^{-}_T(i, j))
	\label{gammaij}
	\end{equation}

Combining Equalities (\ref{dij}), (\ref{gammaij}) and (\ref{sumdeg}), we get

	\begin{equation}
		\gamma_{ij} = 2n - 3 - (d^{+}_T(i) + d^{+}_T(j) + 2d^{-}_T(i, j))
		\label{gamma}
	\end{equation}
		
	By double-counting principle, we obtain	
		\begin{equation}
	\sum_{i<j}\gamma_{ij} = 	\sum_{k}{d^{+}_T(k)}{d^{-}_T(k)}
	\label{gammaxy}
	\end{equation}

		Using Equalities (\ref{sumdeg}) and (\ref{gammaxy}), we get 
	
		\[
		  \sum_{i<j}\gamma_{ij} = (n-2)\sum_{k}  d^{+}_T(k) - 2 \binom{ d^{+}_T(k)}{2}
		\]
		
		It follows from Equalities (\ref{dx}) and (\ref{c3T}) that
			\[
							c_3(T)=\frac{1}{2}\sum_{i<j}\gamma_{ij} -\frac{1}{12} \, {\left(n - 1\right)} {\left(n - 2\right)} n						
			\]
		
		Equality (\ref{mij}) allows to complete the proof of Proposition 5.

\section{Characterization of $n$-tournaments with maximum number of diamonds for $n\equiv0,2,3\pmod{4}$}
Throughout this section, $T$ denotes an $n$-tournament and $S$ denotes its Seidel adjacency matrix. 

It follows from Proposition \ref{sumentries} that
\begin{equation}\label{0mod4}
	\delta_T\leq \frac{1}{96}n^2(n-1)(n-2)
\end{equation}
Equality holds if and only if $m_{ij}=0$ for every $i\neq j$. Since $m_{ii}=1-n$, then $\delta_T = \frac{1}{96}n^2(n-1)(n-2)$ if and only if $S^2=(1-n)I_n$, or equivalently $SS^t = (n-1)I_n$.  Skew-symmetric Seidel matrices that satisfy this equality are called \emph{skew-conference} matrices, and exist only if $n$ is divisible by $4$.

\begin{example}
			For any prime power $q\equiv3\pmod{4}$, the Paley tournament $T(q)$ is the
tournament whose vertices are elements of $\mathbb{F}_q$ where the vertex $x$ dominates the vertex $y$ iff
$y-x$ is a square in $\mathbb{F}_q$. Let $T^{*}(q)$ be the tournament on $n=q+1$ vertices obtained by adding to $T(q)$ a new vertex which dominates all vertices of $T(q)$.
It is well-known that the Seidel adjacency matrix of $T^{*}(q)$ is a skew conference matrix. Then, the number of diamonds in $T^{*}(q)$ is $\frac{1}{96}n^2(n-1)(n-2)$. 
\end{example}

For $n$ odd, we obtain the following refinement of Equality (\ref{0mod4}).
\begin{proposition} \label{diamonds3mod4}
	If $n$ is odd, then
	\[ \delta_T \leq \frac{1}{96}n(n-1)(n-3)(n+1) \]
	Moreover, equality holds if and only if	 $|S^2 + n I_n| = J_{n}$.
\end{proposition}

\begin{proof}
			
Let $S^2:=(m_{ij})$.
Since $n$ is odd, by Equality (\ref{mij}), $m_{ij}$ is also odd.
Hence $\sum_{i<j}{m_{ij}^2} \geq n(n-1)/{2}$.
	
By Proposition \ref{sumentries}, 
	\[\displaystyle
	\begin{array}{lll}
	      \delta_{T}  & = & \frac{1}{96}n^2(n-1)(n-2) - \frac{1}{16}\sum_{i<j}{m_{ij}^2}\\
				            &\leq & \frac{1}{96}n^2(n-1)(n-2) - \frac{1}{16} n(n-1)/{2}\\
										& = & \frac{1}{96}n(n-1)(n-3)(n+1)

	\end{array}\]

Equality holds if and only if   $|m_{ij}| = 1$   for $i \neq j$, or equivalently $|S^2 + n I_n| = J_{n}$, because $m_{ii}=1-n$.
\end{proof}

\begin{theorem}
	If $n \equiv 3 \pmod{4}$, then the following assertions are equivalent
	\begin{enumerate}[label=\arabic*.]
		\item $T$ has $\frac{1}{96}n(n-1)(n-3)(n+1)$ diamonds.
		\item There exists a diagonal ${\pm 1}$-matrix D such that $DS^2D+nI_n=J_n$.
		\item $T$ is switching equivalent to a doubly regular tournament.	
	\end{enumerate}
\end{theorem}

\begin{proof}

$1.\Leftrightarrow 2.$	Assume that $T$ has $\frac{1}{96}n(n-1)(n-3)(n+1)$ diamonds. By Proposition \ref{diamonds3mod4}, we have 
 $m_{ij} = \pm1$ for every $i\neq j$. Using Equalities (\ref{mij}) and (\ref{gamma}), we get:

\begin{enumerate}
			\item[i.] If $ d^{+}_T(i) \equiv d^{+}_T(j) \pmod{2}$, then $\gamma_{ij} \equiv 1 \pmod{2}$, and $m_{ij} \equiv 1 \pmod{4}$. Hence $m_{ij} = 1$.
			\item[ii.] If $ d^{+}_T(i) \equiv d^{+}_T(j) + 1 \pmod{2}$, then $\gamma_{ij} \equiv 0 \pmod{2}$, and $m_{ij} \equiv -1 \pmod{4}$. Hence $m_{ij} = -1$.
	\end{enumerate}

	Let $D = diag(\epsilon_1, \epsilon_2,\ldots, \epsilon_n)$ be the diagonal matrix such that $\epsilon_i = 1$ if $d^{+}_T(i)$ is even and $\epsilon_i = -1$ otherwise. It is easy to check that $\epsilon_i m_{ij} \epsilon_j = 1$ for all $i\neq j$. Then, $DS^2D + nI_n =  J_n$. 
	
	The converse is trivial.

 $2.\Leftrightarrow 3.$ 

Suppose that there exists a diagonal ${\pm 1}$-matrix D such that $DS^2D+nI_n=J_n$. The tournament $T^{'}$ whose Seidel adjacency matrix is $DSD$ is switching equivalent to $T$. We will prove that $T^{'}$ is doubly regular.
							
Let $DS^2D := (m_{ij}^{'})$. Thus, for $i\neq j\in \{1,\ldots,n\}$, we have					
											\[ m_{ij}^{'} = 1 \]
											
By Proposition \ref{cmij}, we get
    \[ c_3(T^{'})=\frac{1}{24}n(n-1)(n+1) \]					

It follows from Remark \ref{c3reg} that the tournament $T^{'}$ is regular.						

	By Identities (\ref{mij}) and  (\ref{gamma}), the in-degree of each pair in $T^{'}$ is $(n-3)/4$. Hence, by definition, $T^{'}$ is doubly regular.
	
	Conversely, assume that $T$ is switching equivalent to a doubly regular tournament $T^{'}$. The out-degree and the in-degree of each pair $(i,j)$ in $T^{'}$ is $(n-3)/4$.  Let ${S^{'}}$ be the Seidel adjacency matrix of $T^{'}$ and let ${S^{'}}^2:= (m^{'}_{ij})$. By Equalities (\ref{mij}) and (\ref{gamma}), we get $m^{'}_{ij}=1$ for $i\neq j\in \{1,\ldots,n\}$ and hence ${S^{'}}^2+nI_n=J_n$. Since $T$ and $T^{'}$ are switching equivalent, $S^{'}=DSD$ for some $\{\pm1\}$-diagonal matrix $D$ and then $DS^2D+nI_n=J_n$.

\end{proof}
	
\begin{theorem}\label{max2}
	If $n \equiv 2 \pmod{4}$, then
	\[ \delta_T \leq \frac{1}{96}n(n-3)(n-2)(n+2) \]
	Moreover, equality holds if and only if
		\[ |PS^2P^t + (n+1)I_n| = \begin{pmatrix}
														2J_{n/2} & O \\
														  O      & 2J_{n/2}
													\end{pmatrix}	 \]
		for some permutation matrix $P$.
\end{theorem}
	
\begin{proof}
	We label the vertices of $T$ so that the first $c$ vertices have an even out-degree and the remaining $n-c$ vertices have an odd out-degree.
  With respect to this labelling, the Seidel adjacency matrix of $T$ is $PSP^t$ where $P$ is a permutation matrix. Let $PS^2P^t:=(m^{'}_{ij})  $.
	Let $i\neq j\in\{1,\ldots,n\}$ such that $ d^{+}_T(i) \equiv d^{+}_T(j) \pmod{2} $. By Equality (\ref{gamma}), $\gamma_{ij} \equiv 1 \pmod{2}$, and by Equality (\ref{mij}), $m^{'}_{ij} \equiv 2 \pmod{4}$. Therefore, ${m^{'}_{ij}}^2 \geq 4$.
		
	By Proposition \ref{sumentries}, we have
			\[\begin{array}{lll}
					\delta_{T}  & =   & \displaystyle\frac{1}{96}n^2(n-1)(n-2) - \frac{1}{16}\sum_{i<j}{{m^{'}_{ij}}^2}\\
											&\leq & \displaystyle\frac{1}{96}n^2(n-1)(n-2) - \frac{1}{16}\left(\sum_{1\leq i<j\leq c}{{m^{'}_{ij}}^2}+\sum_{c+1\leq i<j\leq n}{{m^{'}_{ij}}^2}\right)\\
											&\leq & \displaystyle\frac{1}{96}n^2(n-1)(n-2) - \frac{1}{16}\left(4\binom{c}{2}+4\binom{n-c}{2}\right)\\
											&\leq & \displaystyle\frac{1}{96}n^2(n-1)(n-2) - \frac{4}{16}\left(\frac{n}{2}\left(\frac{n}{2}-1\right)\right)=\frac{1}{96}n(n-3)(n-2)(n+2)
											
					\end{array}
			\]
	Equality holds iff $c=\frac{n}{2}$, $m^{'}_{ij}=\pm 2$ if $1\leq i<j\leq c$ or $c+1\leq i<j\leq n$ and $m^{'}_{ij}=0$ otherwise.
\end{proof}

We give two classes of tournaments that satisfy the conditions of Theorem \ref{max2}.

\begin{enumerate}[leftmargin=*]
	\item 
Recall that an EW matrix $B$ of order $n\equiv2\pmod{4}$ is a $(\pm1)$-matrix verifying $BB^t = B^tB = \begin{pmatrix}
  					M & 0 \\
  					0 & M 
				 \end{pmatrix}$ where $M=(n-2)I_{n/2}+2J_{n/2}$. Ehlich \cite{ehlich64} and Wojtas \cite{wojtas64} independently proved that EW matrices have the maximum determinant among $\pm1$-matrices of order $n\equiv 2 \pmod 4$. EW matrices exist only if $2n - 2$ is the sum of two squares. An EW matrix is said to be of \emph{skew type} if $B+B^t=2I_n$. Such matrix exist only if $2n-3$ is a square, hence, there are no EW matrices of skew type with order $n=10, 18, 22, 30, 34, 38, 46, 50$.
				 
Consider the matrix $B - I_n$, where $B$ is an EW-matrix of skew type. Clearly, this matrix is skew-symmetric, moreover, it has the maximum determinant among skew-symmetric Seidel matrices of order $n\equiv2\pmod4$ \cite{armario16}. By simple computation, we have
\[(B - {I_n})^2 = \begin{pmatrix}
  			    				-2J_{n/2} & O \\
					             O      & -2J_{n/2}
						   \end{pmatrix}+(3-n)I_n.\]
Hence, by Theorem \ref{max2}, if $S+I_n$ is an EW matrix, then $T$ has $\frac{1}{96}n(n-3)(n-2)(n+2)$ diamonds.
Moreover, by ~\cite[Lemma 3.3]{greaves17}, the characteristic polynomial of $S$ is $P_S(x) = (x^2 + 8k + 1)(x^2 + 4k-1)^{2k}$ where $n=4k+2$.

\item Let $T$ be a doubly regular tournament on $n+1=4k+3$ vertices and let $T^{'}$ be the tournament obtained by removing any vertex $v$ of $T$. It is easy to see that for every vertices $i, j$ of $T^{'}$, we have
\[
  d^{+}_{T^{'}}(i, j)+d^{-}_{T^{'}}(i, j)=
	\begin{cases}
		2k-1 & \mbox{if } v\in N^{+}_T(i)\cap N^{+}_T(j)\mbox{ or } v\in N^{-}_T(i)\cap N^{-}_T(j) \\
		2k & \mbox{otherwise} 
	\end{cases}
\]

Let $S$ be the Seidel adjacency matrix of $T^{'}$. Using Identities (\ref{mij}) and (\ref{gammaij}), we find that up to permutation \[S^2+(n+1)I_{n}=\begin{pmatrix}
  			    				2J_{n/2} & O \\
					                    O & 2J_{n/2}
						   \end{pmatrix}.\] Hence, again by Theorem \ref{max2}, the tournament $T^{'}$ has  $\frac{1}{96}n(n-3)(n-2)(n+2)$ diamonds. Moreover, by~\cite[Lemma~4.2.iii]{greaves17}, $P_S(x) = (x^2+1)(x^2+4k+3)^{2k}$.			  
\end{enumerate}

\begin{remark}
Let $T$ be a tournament with $n=4k+2$ vertices and let $S$ be its Seidel adjacency matrix. It follows from~\cite[Lemmata 3.3 and 3.7]{greaves17} that the characteristic polynomial of $S$ is $P_S(x) = (x^2 + 8k + 1)(x^2 + 4k-1)^{2k}$ iff there is a $\pm 1$-diagonal matrix D such that $DSD$ is a skew-symmetric EW matrix.
\end{remark}
	
\begin{remark}
Let $T$ be a tournament and let $S$ be its Seidel adjacency matrix. It follows from Lemma \ref{diam principa minors} and Equality (\ref{eq01}) that if $P_S(x) = (x^2 + 8k + 1)(x^2 + 4k-1)^{2k}$ or $P_S(x) = (x^2+1)(x^2+4k+3)^{2k}$, then $T$ has the maximum number of diamonds.
\end{remark}

Up to switching, there are two 6-tournaments with the maximum number of diamonds $6$, one of them is obtained by removing a vertex from a doubly regular tournament, and the other consists of two $3$-cycles one dominating the other, its Seidel adjacency matrix is a skew symmetric EW matrix.

As for $10$-tournaments, using SageMath, we found two switching classes of tournaments with the maximum number of diamonds $70$. The characteristic polynomial of the tournaments in the first class is $(x^2+1)(x^2+11)^4$, we identified one as a tournament obtained by removing a vertex from a doubly regular tournament.
The characteristic polynomial of the tournaments in the second class is $(x^2+1)(x^4+18x^2+61)^2$. The tournament with the following Seidel adjacency matrix belongs in the second class. 
\[
		S=\left(\begin{array}{rrrrrrrrrr}
									0 & -1 & 1 & -1 & 1 & -1 & -1 & 1 & -1 & -1 \\
									1 & 0 & -1 & -1 & -1 & -1 & 1 & 1 & 1 & -1 \\
									-1 & 1 & 0 & -1 & -1 & 1 & -1 & 1 & -1 & -1 \\
									1 & 1 & 1 & 0 & -1 & -1 & 1 & -1 & -1 & -1 \\
									-1 & 1 & 1 & 1 & 0 & -1 & -1 & -1 & 1 & -1 \\
									1 & 1 & -1 & 1 & 1 & 0 & -1 & -1 & -1 & -1 \\
									1 & -1 & 1 & -1 & 1 & 1 & 0 & -1 & 1 & -1 \\
									-1 & -1 & -1 & 1 & 1 & 1 & 1 & 0 & -1 & -1 \\
									1 & -1 & 1 & 1 & -1 & 1 & -1 & 1 & 0 & -1 \\
									1 & 1 & 1 & 1 & 1 & 1 & 1 & 1 & 1 & 0
						\end{array}
			\right)
\]
Curiously, $S$ has the maximum determinant among $\pm1$ matrices, but $S+I_n$ is not an EW matrix. This leads to the following questions.

\begin{question}
	Let $S$ be a skew Seidel matrix with the maximum determinant, does its corresponding tournament have the maximum number of diamonds ?
\end{question}

The answer to this question is positive in the following two cases:
\begin{enumerate}
\item There exists a skew conference matrix of order $n\equiv0\pmod{4}$.
\item There exists a skew EW matrix of order $n\equiv2\pmod{4}$.
\end{enumerate}
\begin{question}
		Let $n=4k+2$. Is there an infinite family of $n$-tournaments with $\frac{1}{96}n(n-3)(n-2)(n+2)$ diamonds, such that their Seidel adjacency matrices has a characteristic polynomial that is neither $(x^2 + 8k + 1)(x^2 + 4k-1)^{2k}$ nor $(x^2+1)(x^2+4k+3)^{2k}$.
\end{question}

\section{The case of $n\equiv1\pmod{4}$}	
			  
We start with following lemma.
\begin{lemma}
	Let $T=(V, A)$ be a tournament on $n+1 \equiv{2}\pmod{4}$ vertices with $\frac{1}{96}(n+1)(n-2)(n-1)(n+3)$ diamonds. 
	Then
	\[ \delta_{T-v} = \frac{1}{96}  {\left(n + 3\right)} {\left(n - 1\right)} {\left(n - 2\right)} {\left(n - 3\right)}
 \]
	for every $v\in V$.
\end{lemma}

\begin{proof}
		Let $T=(V,A)$ be a tournament on $n+1 = 4k+2$ vertices containing $\frac{1}{96}(n+1)(n-2)(n-1)(n+3)$ diamonds. 
		Fix a vertex of $v$ and let $T^{'}$ be the tournament switching equivalent to $T$ in which $v$ dominates the other vertices. 		
		
By adapting the proof of  \cite[Lemma~2.1]{armario15}, we show easily that the the score vector of $T^{'}-v$ is $2k$, $2k+1$ and $2k-1$, each
appearing $2k+1$ , $k$ and $k$ times, respectively.

By Equality (\ref{c3T}), the number of $3$-cycles in $T^{'}-v$ is

\[c_3(T^{'}-v) = \binom{4k+1}{3}- \left((2k+1)\binom{2k}{2}+k\binom{2k-1}{2}+k\binom{2k+1}{2}\right) \]

That is,

	\[
	c_3(T^{'}-v) =  \frac{1}{24}  {\left(n +3\right)} {\left(n - 1\right)} {\left(n - 2\right)}
  \]
	
	Since $v$ dominates every vertex in $V\setminus \{v\} $ in the tournament $T^{'}$, by Lemma \ref{deltac3}
	
	\[
		\delta_{T^{'}-v} = \delta_{T^{'}} - c_3(T^{'}-v)
	\]	
	
	Hence,
	
				\[ \delta_{T^{'}-v} = \frac{1}{96} \, {\left(n + 3\right)} {\left(n - 1\right)} {\left(n - 2\right)} {\left(n - 3\right)}
        \]
It follows that
\[\delta_{T-v} = \frac{1}{96} \, {\left(n + 3\right)} {\left(n - 1\right)} {\left(n - 2\right)} {\left(n - 3\right)}
 \]
because $T^{'}-v$ is switching equivalent to $T-v$.
\end{proof}

The previous lemma leads to the following conjecture.
\begin{conjecture}
  Let $T=(V, A)$ be a tournament on $n\equiv 1\pmod4$ vertices. Then, the number of diamonds in $T$ is at most 
	\[\frac{1}{96}  {\left(n + 3\right)} {\left(n - 1\right)} {\left(n - 2\right)} {\left(n - 3\right)}\]

\end{conjecture}

The following lemma gives another way to obtain $n$-tournaments with $\frac{1}{96}  {\left(n + 3\right)} {\left(n - 1\right)} {\left(n - 2\right)} {\left(n - 3\right)}$ diamonds. 
\begin{lemma}
	Let $T=(V, A)$ be a tournament on $n-1$ vertices such that its Seidel adjacency matrix is a skew-conference matrix . Let $T^{'}$ be a tournament obtained from $T$ by adding a vertex that dominates all vertices in $V$, then
		\[ 
			\delta_{T^{'}} = \frac{1}{96}  {\left(n + 3\right)} {\left(n - 1\right)} {\left(n - 2\right)} {\left(n - 3\right)}
	  \]

\end{lemma}

\begin{proof}
			Let $T^{'}$ be a tournament obtained from $T$ by adding a vertex $v$ that dominates all vertices in $V$, the number of diamonds in the tournament $T^{'}$ is
								\[
										\delta_{T^{'}} = \delta_{T} + c_3(T)
									\]
		Let $S$ be the Seidel adjacency matrix of $T$. By Propositions \ref{sumentries} and \ref{cmij}, we have
						 \begin{eqnarray*}
									c_3(T)     & = & \frac{1}{24}(n-1)(n-2)(n-3) +\frac{1}{4}\sum_{i<j}{m_{ij}} 	\\
									\delta_{T} & = & \frac{1}{96}(n-1)^2(n-2)(n-3) - \frac{1}{16}\sum_{i<j}{m_{ij}^2}
			       \end{eqnarray*}
Since the non-diagonal entries of $S^2$ are equal to zero, then
				\[
				\delta_{T^{'}} = \frac{1}{96} \, {\left(n + 3\right)} {\left(n - 1\right)} {\left(n - 2\right)} {\left(n - 3\right)}
				\]
							
\end{proof}

\bibliography{bibpaper}

\end{document}